\documentclass[11pt]{amsart}

\usepackage{amssymb,amsmath,amsthm,newlfont,enumerate}

\theoremstyle{theorem}
\newtheorem{theorem}{Theorem}[section]
\newtheorem{lemma}[theorem]{Lemma}
\newtheorem{proposition}[theorem]{Proposition}
\newtheorem{corollary}[theorem]{Corollary}

\theoremstyle{definition}

\numberwithin{equation}{section}

\newcommand{\CC}{{\Bbb C}}
\newcommand{\DD}{{\Bbb D}}
\newcommand{\HH}{{\Bbb H}}
\newcommand{\RR}{{\Bbb R}}
\newcommand{\TT}{{\Bbb T}}
\newcommand{\cL}{{\mathcal L}}

\let\Re\undefined

\DeclareMathOperator{\Re}{\mathrm Re}

\DeclareMathOperator{\supp}{\mathrm supp}
\DeclareMathOperator{\conv}{\mathrm conv}

\begin{document}

\title[Titchmarsh convolution theorem]{Letter to the editor: \\A short complex-variable proof of the Titchmarsh convolution theorem}

\author{Thomas Ransford}
\address{D\'epartement de math\'ematiques et de statistique, Universit\'e Laval, 
Qu\'ebec City (Qu\'ebec),  Canada G1V 0A6}
\email{thomas.ransford@mat.ulaval.ca}
\thanks{Research supported by grants from NSERC and the Canada Research Chairs program}

\begin{abstract}
The Titchmarsh convolution theorem is a celebrated result
about the support of the convolution of two functions.
We present a simple proof based on the canonical factorization theorem for 
bounded holomorphic functions
on the unit disk.
\end{abstract}



\maketitle

\section{Introduction}

Let $f,g:\RR\to\CC$ be  integrable functions, 
and let $f*g$ be their convolution product, namely
\[
(f*g)(t):=\int_{-\infty}^\infty f(t-s)g(s)\,ds 
\quad(t\in\RR).
\]
It is clear that if $f=0$ a.e.\ on $(-\infty,a)$
and if $g=0$ a.e.\ on $(-\infty, b)$, then
$f*g=0$ a.e.\ on $(-\infty,a+b)$. Hence,
defining
\[
\alpha(f):=\sup\{a\in\RR: f=0 \text{~a.e. on~}(-\infty,a)\},
\]
and likewise for $g$ and $f*g$,
we have $\alpha(f*g)\ge \alpha(f)+\alpha(g)$.
Much less obvious is the following theorem.

\begin{theorem}\label{T:Titchmarsh}
If $\alpha(f)>-\infty$ and $\alpha(g)>-\infty$, then
\[
\alpha(f*g)=\alpha(f)+\alpha(g).
\]
\end{theorem}

Theorem~\ref{T:Titchmarsh} is a version of the Titchmarsh convolution theorem \cite{Ti26}.
This celebrated result  has applications in a number of fields, including
harmonic analysis, Banach algebras,  operator theory and partial differential equations.

There are several proofs of the Titchmarsh convolution theorem, none of them easy.
Broadly speaking, they divide into two groups:
those using real-variable methods 
(see e.g.\ \cite{Dos88, Kal62, Mi53, Mi59}),
and those based on complex-variable techniques 
(Titchmarsh's original proof was of this kind, 
see also \cite{Bo54, Cr41, Du47, Ko57, La59}).
The complex-variable proofs are maybe more natural, 
but tend to require more advanced results from function theory.
In this paper we present a simple proof that uses nothing beyond
the well-known canonical factorization theorem for bounded  holomorphic functions
on the unit disk.

\section{Laplace transforms}

Complex-variable proofs of Titchmarsh's theorem usually
proceed via the Laplace transform,
so we begin by briefly reviewing this notion.
In what follows, we write $\HH$ for the right half-plane,
namely $\HH:=\{z\in\CC:\Re z>0\}$.

Let $f:\RR\to\CC$ be an integrable function with $\alpha(f)>-\infty$.
Its Laplace transform
$\cL f:\HH\to\CC$ is defined by
\[
\cL f(z):=\int_{-\infty}^\infty f(t) e^{-zt}\,dt \qquad(z\in \HH).
\]
It is easy to see that $\cL f$ is holomorphic on $\HH$. 
Furthermore, a straightforward estimate shows that
\[
|\cL f(z)|\le Ae^{-\alpha(f)\Re z} \qquad(z\in \HH),
\]
where  $A:=\int_{-\infty}^\infty|f(t)|\,dt$.
In particular, if $\alpha(f)\ge0$, then $\cL f$ is bounded on~$\HH$. 
The converse is also true:

\begin{lemma}\label{L:Liouville}
If $\cL f$ is bounded on $\HH$, then $\alpha(f)\ge0$.
\end{lemma}

\begin{proof}
Set $f_1:=f1_{(-\infty,0]}$ and $f_2:=f1_{(0,\infty)}$. 
As $f_1$ is supported on a compact subset of $(-\infty,0]$, 
its Laplace transform $\cL f_1$ is an entire function 
which is bounded on the left half-plane $\Re z\le 0$.
On the right half-plane $\Re z>0$, we have $\cL f_1=\cL f-\cL f_2$, 
where $\cL f_2$ is clearly bounded, 
and $\cL f$ is bounded by assumption.
So in fact $\cL f_1$ is bounded on the whole complex plane. 
By Liouville's theorem $\cL f_1$ is constant. 
It is easily seen that $\lim_{x\to-\infty}\cL f_1(x)=0$, 
so the constant must be zero, i.e.\ $\cL f_1\equiv0$.
By the uniqueness theorem for Laplace transforms, $f_1=0$ a.e.
Hence $f=f_2$ a.e., and so, finally, $\alpha(f)=\alpha(f_2)\ge0$.
\end{proof}

Lastly in this section, 
we remark that, if $f,g$ are integrable functions 
such that $\alpha(f),\alpha(g)>-\infty$, 
then $\cL(f*g)=\cL f.\cL g$, the pointwise product on $\HH$.
The proof is a routine argument using Fubini's theorem.

\section{Beginning of the proof of Theorem~\ref{T:Titchmarsh}}

We argue by contradiction.
Suppose, if possible, that there are 
integrable functions $f,g$ with 
$\alpha(f)>-\infty$ and $\alpha(g)>-\infty$
such that $\alpha(f*g)>\alpha(f)+\alpha(g)$.
Replacing $f$ and $g$ by appropriate translates, 
we may further assume that $\alpha(f)<0$ and $\alpha(g)<0$
and $\alpha(f*g)>0$.
By Lemma~\ref{L:Liouville}, the Laplace transforms $\cL f, \cL g$
are unbounded on $\HH$, whereas their product $\cL f.\cL g$ is bounded on~$\HH$.
So, to obtain a contradiction, 
and thereby complete the proof of Theorem~\ref{T:Titchmarsh}, 
it suffices to establish the following general result  
about holomorphic functions.

\begin{proposition}\label{P:mainresult}
Let $F$ and $G$ be holomorphic functions on $\HH$ satisfying
$|F(z)|\le Ae^{a\Re z}$ and $|G(z)|\le Be^{b\Re z}$
for some positive constants $A,a,B,b$.
If both $F$ and $G$ are \mbox{unbounded} on $\HH$, then so is their product $FG$.
\end{proposition}

Up till now, we have followed a fairly standard route. 
The novelty in our approach lies in our proof Proposition~\ref{P:mainresult},
for which we shall use the canonical factorization theorem for bounded holomorphic functions
on the unit disk. We pause briefly to review this theorem.

\section{Canonical factorization theorem}

Let $\DD$ denote the open unit disk and $\TT$ denote the unit circle.

A \emph{Blaschke product} on $\DD$ is a function of the form
\[
B(z):=c z^m\prod\frac{|a_n|}{a_n}\frac{a_n-z}{1-\overline{a}_nz}
\quad(z\in\DD),
\]
where $c$ is a unimodular constant, $m$ is a non-negative integer, and $(a_n)$ is a (finite or infinite)
sequence of points in $\DD$ such that $\sum_n(1-|a_n|)<\infty$.

A \emph{bounded outer function} on $\DD$ is a function of the form
\[
O(z):= \exp\Bigl(\int_\TT \frac{e^{i\theta}+z}{e^{i\theta}-z}\log \omega(\theta)\,\frac{d\theta}{2\pi}\Bigr)
\quad(z\in\DD),
\]
where $\omega:\TT\to(0,\infty)$ is a bounded positive function such that $\log \omega\in L^1(\TT)$.

A \emph{singular inner function} on $\DD$ is a function of the form
\[
S(z):=\exp\Bigl(-\int_\TT \frac{e^{i\theta}+z}{e^{i\theta}-z}\,d\sigma(e^{i\theta})\Bigr)
\quad(z\in\DD),
\]
where $\sigma$ is finite positive Borel measure on $\TT$ that is singular with respect to Lebesgue measure.
The measure $\sigma$ is uniquely determined by $S$. We shall often write $S=S_\sigma$ to indicate the dependence of $S$ upon $\sigma$.

The following result is known as the canonical factorization theorem.
A proof can be found for example in \cite[Theorem~2.8]{Du70}.

\begin{theorem}\label{T:factorization}
Every bounded holomorphic function $H$ on $\DD$ with  $H\not\equiv0$  has a unique factorization of the form $H=BOS$, where $B$ is a Blaschke product, $O$ is a bounded outer function, and $S$ is a singular inner function.
\end{theorem}

The following corollary will prove crucial in what follows.

\begin{corollary}\label{C:factorization}
Let $B$ be a Blaschke product, let $O$ be a bounded outer function, and let $S_{\sigma_1}$ and $S_{\sigma_2}$ be singular inner functions. Then $BOS_{\sigma_1}/S_{\sigma_2}$ is bounded on $\DD$ if and only if $\sigma_1-\sigma_2$ is a positive measure.
\end{corollary}

\begin{proof}
The `if' is obvious, since $BOS_{\sigma_1}/S_{\sigma_2}=BOS_{\sigma_1-\sigma_2}$. 
For the `only if', suppose that
$BOS_{\sigma_1}/S_{\sigma_2}$ is bounded. Then, by Theorem~\ref{T:factorization}, 
we can write it as $\widetilde{B}\widetilde{O}\widetilde{S}$,
where $\widetilde{B}$ is a Blaschke product, $\widetilde{O}$ is a bounded outer function,
and $\tilde{S}$ is a singular inner function, say $\widetilde{S}=S_\sigma$.
Multiplying up by $S_{\sigma_2}$, we then have
$BOS_{\sigma_1}=\widetilde{B}\widetilde{O}S_{\sigma_2+\sigma}$. 
By the uniqueness part of Theorem~\ref{T:factorization},
it follows that $\sigma_1=\sigma_2+\sigma$. Thus $\sigma_1-\sigma_2$ is a positive measure.
\end{proof}

\section{Completion of the proof of Theorem~\ref{T:Titchmarsh}}

\begin{proof}[Proof of Proposition~\ref{P:mainresult}]
By assumption, the function $F(z)e^{-az}$ is bounded on $\HH$.
Composing with $\phi(z):= (1+z)/(1-z)$, which maps $\DD$ conformally onto $\HH$, 
we obtain a bounded holomorphic function on $\DD$. Hence, by Theorem~\ref{T:factorization},
\[
F(\phi(z))e^{-a\phi(z)}=B(z)O(z)S_\sigma(z)
\quad(z \in\DD),
\]
where $B$ is a Blaschke product, $O$ is a bounded outer function,
and $S_\sigma$ is a singular inner function. Notice also that $e^{-a\phi(z)}=S_{a\delta_1}(z)$,
where $\delta_1$ denotes the Dirac mass at the point $1$.
Thus we have
\[
F\circ\phi =BOS_\sigma/S_{a\delta_1}.
\]
If $F$ is unbounded on $\HH$, then $F\circ\phi$ is unbounded on $\DD$,
so by Corollary~\ref{C:factorization} $\sigma-a\delta_1$ is not a positive measure,
in other words $\sigma(\{1\})<a$. 

Likewise, we can write
\[
G\circ\phi=\widetilde{B}\widetilde{O}S_\tau/S_{b\delta_1},
\]
where $\widetilde{B}$ is a Blaschke product, $\widetilde{O}$ is a bounded outer function,
$S_\tau$ is a singular inner function,
and, if $G$ is unbounded, then $\tau(\{1\})<b$. 

Multiplying these equations together, we obtain
\[
(FG)\circ\phi=(B\widetilde{B})(O\tilde{O})S_{\sigma+\tau}/S_{(a+b)\delta_1}.
\]
If $F$ and $G$ are both unbounded, then $(\sigma+\tau)(\{1\})<a+b$, 
so the measure $(\sigma+\tau)-(a+b)\delta_1$ is not positive.
Corollary~\ref{C:factorization} then shows that $(FG)\circ\phi$ is unbounded
on $\DD$, whence $FG$ is unbounded on $\HH$. This completes the proof
of Proposition~\ref{P:mainresult}, and therefore also that of Theorem~\ref{T:Titchmarsh}.
\end{proof}

\section{Concluding remarks}
The Titchmarsh convolution theorem can be expressed in terms of supports. The 
\emph{support} an integrable function $f:\RR\to\CC$ is 
the smallest closed subset of $\RR$ such that $f=0$ a.e.\ on the complement. 
Writing $\supp(f)$ for this support, it is easy to see that, for all integrable functions $f,g$, we have
\begin{equation}\label{E:supp}
\supp(f*g)\subset\supp(f)+\supp(g).
\end{equation}
If further $\supp(f)$ and $\supp(g)$ are both compact,
then Theorem~\ref{T:Titchmarsh} tells us that the two sides
of \eqref{E:supp} have the same minimum and (after  reflection) the same maximum. Thus,
writing $\conv(\cdot)$ for the convex hull, we have
\begin{equation}\label{E:conv}
\conv(\supp(f*g))=\conv(\supp(f))+\conv(\supp(g)).
\end{equation}

In this form, the result generalizes to  higher dimensions.
If $f,g:\RR^n\to\CC$ are integrable functions with compact
supports, then \eqref{E:conv} holds. 
More generally still, \eqref{E:conv} holds whenever $f$
and $g$ are distributions on $\RR^n$ with compact supports. This generalization is due to Lions \cite{Li53}.
It can be deduced quite easily from the basic version, Theorem~\ref{T:Titchmarsh} (see e.g.\ \cite[\S45]{Don69}).

\bibliographystyle{amsplain}
\bibliography{biblist}

\end{document}